\documentclass[12pt]{article}
\usepackage{amssymb,amsmath,graphicx,acronym,amsfonts,float,color,epstopdf}
\usepackage{algorithmicx,algorithm}
\usepackage{graphicx}
\usepackage{subfigure}
\usepackage{caption}
\usepackage{graphicx,subfig}
\newcommand{\qed}{\bull \medskip}

\newtheorem{theorem}{Theorem}[section]

\newtheorem{lemma}{Lemma}[section]
\newtheorem{example}{Example}[section]
\newtheorem{definition}{Definition}
\newtheorem{assum}{Assumption}[section]

\def\bt{\begin{theorem}}
\def\et{\end{theorem}\bigskip}
\def\bl{\begin{Lemma}}
\def\el{\end{Lemma}\bigskip}
\def\ep{\end{Proposition}\bigskip}
\def\bp{\begin{Proposition}}
\def\bd{\begin{definition}}
\def\ed{\end{definition}}
\definecolor{red8}{rgb}{0.8,0.0,0.8}
\definecolor{red5}{rgb}{0.5,0.0,0.5}
\definecolor{red6}{rgb}{0.6,0.0,0.6}
\definecolor{red7}{rgb}{0.7,0.0,0.7}
\definecolor{red9}{rgb}{0.9,0.0,0.9}
\definecolor{green2}{rgb}{0.0,0.2,0.0}
\definecolor{green3}{rgb}{0.0,0.3,0.0}
\definecolor{green4}{rgb}{0.0,0.4,0.0}
\definecolor{blue4}{rgb}{0.4,0.0,1.0}
\definecolor{yellow4}{rgb}{0.9,0.8,0.5}
\definecolor{blue5}{rgb}{0.0,0,0.5}
\definecolor{blue6}{rgb}{0.0,0.5,0.7}
\definecolor{hong1}{rgb}{0.8,0.8,0.8}

\def\qed{\hfill {$\Box$} \medskip}
\def\proof{\noindent\bf Proof. \hspace{4mm}\rm}

\newcommand\diff{\,{\mathrm d}}

\input epsf
\begin{document}

\title{\bf Stochastic absolute value equations}

\author{Shouqiang Du\footnote{Corresponding author. School of Mathematics and Statistics,
Qingdao University, Qingdao,  266071, China. E-mail: sqdu@qdu.edu.cn.}\qquad Jingjing Sun\footnote{School of Mathematics and Statistics,
Qingdao University, Qingdao,  266071, China.}\qquad Shengqun Niu\footnote{School of Mathematics and Statistics,
Qingdao University, Qingdao,  266071, China.}\qquad Liping Zhang\footnote{Department of Mathematical Sciences,  Tsinghua University, Beijing 100084, China. E-mail: lipingzhang@tsinghua.edu.cn.}}
\date{}
\maketitle

\begin{abstract} We propose a new kind of stochastic absolute value equations involving absolute values of variables. By utilizing an equivalence relation to stochastic bilinear program, we investigate the expected value formulation for the proposed stochastic absolute value equations. We also consider the expected residual minimization formulation for the proposed stochastic absolute value equations. Under mild assumptions, we give the existence conditions for the solution of the stochastic absolute value equations. The solution of the stochastic absolute value equations can be gotten by solving the discrete minimization problem. And we also propose a smoothing gradient method to solve the discrete minimization problem. Finally, the numerical results and some discussions are given.



\end{abstract}

\newpage

\section{Introduction}

Let $(\Omega, {{\cal F}}, \rho)$ be a probability space, where $\Omega\subseteq R^n$ and $\rho$ is a standard probability measure, we propose a new kind of stochastic absolute value equations, which is to find a vector $x\in R^n$ such that
\begin{equation}\label{e11}A(\omega)x-|x|=b(\omega),\end{equation}
where $A(\omega)\in R^{n\times n}$ and $b(\omega)\in R^n$ for $\omega\in \Omega$  are random quantities on a probability space $(\Omega, {{\cal F}}, \rho)$, $|x|$ is the componentwise absolute value of vector $x\in R^n$. We call (\ref{e11}) the stochastic absolute value equations (SAVE). When $A(\omega)$ is a deterministic matrix and $b(\omega)$ is a deterministic vector, then SAVE (\ref{e11}) reduces to the absolute value equation (AVE) which is equivalent to the general linear complementarity problem \cite{Cot12,mang07,mang06,mang0702}. The AVE  was widely used in solving linear programs, bimatrix games and fundamental problems of mathematical programming, one can see \cite{mang07,mang06,mang0702,cacc11}. In the past few decades, the stochastic variational inequality problems \cite{gur94,rav17}, the stochastic linear complementarity problems \cite{chen05,chen09,fang07,zhang09,zhang11}, the stochastic nonlinear complementarity problems \cite{lin06,linc08} and the stochastic tensor complementarity problems \cite{du20,ming20,du22} were also widely studied in solving many optimization problems with uncertainty. However, no attention has been paid to SAVE (\ref{e11}) which contains the characteristics of AVE and stochastic optimization problems.

As the AVE is an NP hard problem \cite{mang07}, it is also a hard work to solve SAVE (\ref{e11}). Generally, for the stochastic optimization problems, there are two general approaches to get the solution of the problems \cite{chen05,chen09,fang07}. The first approach applies the expected value (EV) method which formulates the problem as a deterministic problem by taking the expect of the stochastic quantity, and the second approach is the expected residual minimization (ERM) method, which is a natural extension of the least-squares method of minimizing the residual. In this paper, the equivalent relation between SAVE (\ref{e11}) and stochastic bilinear program is given. By using the EV formulation, we propose an expected value formulation for SAVE (\ref{e11}). We also study the ERM formulation for SAVE (\ref{e11}). We generate samples by the quasi-Monte Carlo methods and prove that every accumulation point of the discrete approximation problem is the solution of the expected residual minimization problem for SAVE (\ref{e11}).

 The remainder of this paper is organized as follows. In Section 2, we show that SAVE (\ref{e11}) is equivalent to a stochastic bilinear program, which is a stochastic optimization problem with the formula as a stochastic generalized linear complementarity problem. Combined  with an example, we give a discussion about the EV formulation. In Section 3, we first establish the boundedness of the solution set of the expected residual minimization problem, and then show that each accumulation point of the sequence generated by the ERM formulation is a solution of the expected residual minimization problem. In Section 4, we propose a smoothing gradient method for solving SAVE (\ref{e11}). Some numerical experiments are also given to verify the theoretical results of the ERM formulation. Finally, we complete our paper with some conclusions in Section 5.

\setcounter{proposition}{0} \setcounter{theorem}{0}
\setcounter{lemma}{0} \setcounter{corollary}{0}
\setcounter{equation}{0}
\section{Expected value formulation}
We start by showing that SAVE (\ref{e11}) is equivalent to a stochastic bilinear program. By the equivalence of the stochastic bilinear program and the stochastic generalized linear complementarity problem, SAVE (\ref{e11}) can be reformulated as a stochastic generalized linear complementarity problem. Then the expected value formulation will be used to solve SAVE (\ref{e11}).

\begin{theorem}\label{thm21} SAVE (\ref{e11}) is equivalent to the stochastic bilinear program, i.e.,
$$0=\min_{x\in R^{n}}\{((A(\omega)+I)x-b(\omega))^{T}((A(\omega)-I)x-b(\omega))~|~(A(\omega)+I)x-b(\omega)\geq0,~(A(\omega)-I)x-b(\omega)\geq0\}.$$
\end{theorem}
\begin{proof} By SAVE (\ref{e11}), from $|x|\leq A(\omega)x-b(\omega)$, we have
$$(A(\omega)+I)x-b(\omega)\geq0,~(A(\omega)-I)x-b(\omega)\geq0,$$
i.e., the above formulations are the equivalence of the constraints for the stochastic bilinear program. So we have
$$|x|=A(\omega)x-b(\omega)$$
$$\Updownarrow$$
$$((A(\omega)+I)x-b(\omega))^{T}((A(\omega)-I)x-b(\omega))=0, (A(\omega)+I)x-b(\omega)\geq0,~(A(\omega)-I)x-b(\omega)\geq0$$
We complete the proof.
\end{proof}

\begin{theorem}\label{thm22}
SAVE (\ref{e11}) is equivalent to the stochastic generalized linear complementarity problem, i.e.,
\begin{equation}\label{eqn21}
\begin{split}
(A(\omega)+I)x-b(\omega)\geq0,~(A(\omega)-I)x-b(\omega)\geq0,\\ ((A(\omega)+I)x-b(\omega))^{T}((A(\omega)-I)x-b(\omega))=0.
\end{split}
\end{equation}
\end{theorem}
\begin{proof}
By the equivalence of the stochastic bilinear program and the stochastic generalized linear complementarity problem, we get this theorem.
\end{proof}

In the following of this paper, $E[\cdot]$ stands for the expectation of every elements of matrix and vector. $\bar{A}$ denotes the expectation of $A(\omega)$ and $\bar{b}$ denotes the expectation of $b(\omega)$,  i.e.,
$$\bar{A}=E[A(\omega)],~\bar{b}=E[b(\omega)].$$
Then, we get the expected value formulation of the stochastic generalized linear complementarity problem as
\begin{equation}\label{eqn22}
\begin{split}
((\bar{A}+I)x-\bar{b})^{T}((\bar{A}-I)x-\bar{b})=0,\\
(\bar{A}+I)x-\bar{b}\geq0,~(\bar{A}-I)x-\bar{b}\geq0.
\end{split}
\end{equation}
In general, the solution set of (\ref{eqn21}) is not equivalent to the solution set of (\ref{eqn22}) for all $\omega\in\Omega$. So, in this section, we consider a kind of discrete probability space, which has only finitely many elements, i.e.,
$\Omega=\{\omega_{1},\omega_{2},\cdots,\omega_{m}\}$. Now, (\ref{eqn21}) is equivalent to
$$\left\{
\begin{aligned}\label{eqn23}
G(x)\geq0,~H(x)\geq0,~G(x)^{T}H(x)=0,\\
(A(\omega_{i})+I)x-b(\omega_{i})\geq0,\\
(A(\omega_{i})-I)x-b(\omega_{i})\geq0,
\end{aligned}
\right.\eqno(2.3)$$
where $G(x)=(\bar{A}+I)x-\bar{b}$, $H(x)=(\bar{A}-I)x-\bar{b}$, $i=1,2,\cdots,m$.

In the following of this section, we reformulate (2.3) as a nonlinear equations with nonnegative constraints, i.e., the expected value formulation of SAVE (\ref{e11}). (\ref{eqn22}) is a generalized linear complementarity problem, and it can be reformulated as a semismooth equations by Fischer-Burmeister (FB) function.
FB function is an NCP function \cite{Cot12}, which is defined as
$$\phi_{FB}(a,b)=\sqrt{a^{2}+b^{2}}-a-b,$$
where $a,b\in R$. Then $x$ is a solution of (2.3) if and only if $$\tilde{H}(x,y)=0,~y\geq0,\eqno(2.4)$$
where
$$\tilde{H}(x,y)=\begin{pmatrix}
   \Phi(x) \\
   (A(\omega_{1})+I)x-b(\omega_{1})-y_{1}\\
   (A(\omega_{1})-I)x-b(\omega_{1})-y_{2}\\
   \vdots\\
   (A(\omega_{m})+I)x-b(\omega_{m})-y_{2m-1}\\
   (A(\omega_{m})-I)x-b(\omega_{m})-y_{2m}\\
\end{pmatrix}
$$
and
$$\Phi(x)=\begin{pmatrix}
   \phi_{FB}((\bar{A}+I)x-\bar{b})_{1},((\bar{A}-I)x-\bar{b})_{1} \\
   \vdots\\
   \phi_{FB}((\bar{A}+I)x-\bar{b})_{n},((\bar{A}-I)x-\bar{b})_{n}
\end{pmatrix},
$$
with
$$y=(y_{1}^{T},y_{2}^{T},\cdots,y_{2m}^{T})^{T},~y_{i}\in R^{n}, ~i=1,2,\cdots,2m.$$

Now, we give a simple example to illustrate the transformation process.
\begin{example}~Consider SAVE (\ref{e11}), where
\begin{center}
$A(\omega)=\begin{pmatrix}10+\omega&1&2&0\\1&11+\omega&3&1\\0&2&12+\omega&1\\1&7&0&13+\omega\end{pmatrix}$, $b(\omega)=\begin{pmatrix}12+\omega\\15+\omega\\14+\omega\\20+\omega\end{pmatrix}$,
\end{center}
$\Omega=\{\omega_1,\omega_2\}=\{0,2\},P_i=\mathcal{P}(\omega_i\in\Omega)=\frac{1}{2}$.
\end{example}
We know that $(1, 1, 1, 1)^T$ is the solution of this example. Now, we use the EV formulation to solve the above example. Firstly, we get
\begin{center}
$\overline{A}=\begin{pmatrix}11&1&2&0\\1&12&3&1\\0&2&13&1\\1&7&0&14\end{pmatrix}$,
$\overline{b}=\begin{pmatrix}13\\16\\15\\21\end{pmatrix}$.
\end{center}
Then by (2.4), we know that Example 2.1 can be transformed into the following constrained equations
\begin{center}
$\widetilde{H}(x,y)=\begin{pmatrix}\phi_{FB}(12x_1+x_2+2x_3-13,10x_1+x_2+2x_3-13)\\ \phi_{FB}(x_1+13x_2+3x_3+x_4-16,x_1+11x_2+3x_3+x_4-16)\\ \phi_{FB}(2x_2+14x_3+x_4-15,2x_2+12x_3+x_4-15)\\ \phi_{FB}(x_1+7x_2+15x_4-21,x_1+7x_2+13x_4-21)\\11x_1+x_2+2x_3-12-y_1\\x_1+12x_2+3x_3+x_4-15-y_2\\
2x_2+13x_3+x_4-14-y_3\\x_1+7x_2+14x_4-20-y_4\\9x_1+x_2+2x_3-12-y_5\\x_1+10x_2+3x_3+x_4-15-y_6\\2x_2+11x_3+x_4-14-y_7\\
x_1+7x_2+12x_4-20-y_8\\13x_1+x_2+2x_3-14-y_9\\x_1+13x_2+3x_3+x_4-17-y_{10}\\2x_2+15x_3+x_4-16-y_{11}\\
x_1+7x_2+16x_4-22-y_{12}\\11x_1+x_2+2x_3-14-y_{13}\\x_1+12x_2+3x_3+x_4-17-y_{14}\\2x_2+13x_3+x_4-16-y_{15}\\
x_1+7x_2+14x_4-22-y_{16}\end{pmatrix}$,
\end{center}
where $y_i\geq0,i=1,2,\cdots,16$. The optimization solution of the above constrained equations is equivalence to the optimization solution of the following constrained optimization problem
\begin{eqnarray*}
\min && \frac12\|\widetilde{H}(x,y)\|^2\\
s.t. && y\geq0.
\end{eqnarray*}
We use \textbf{fmincon} function in Matlab Optimization Toolbox to solve the transformed constrained optimization problem. The numerical results are given in the following table, where $x^0$ denotes the initial point, $x^*$ denotes the optimum solution.\\

$~~~~~~~~~~~~~~~~~~~~$\textbf{Table 2.1}~~~Numerical results for Example 2.1 \\
$$\begin{tabular}{cccccccccccccc}
\hline\noalign{\smallskip}
	 $x^0$ & $x^*$ &   $\frac12{\parallel \widetilde{H}\parallel}^2$ \\
		\noalign{\smallskip}\hline\noalign{\smallskip}
(2.5127,-2.4490,0.0596,1.9908)$^T$ & (1.000000,1.000000,1.000000,1.000000)$^T$ &2.5580$\times10^{-12}$\\
(-1.4834,3.3083,0.8526,0.4972)$^T$  & (1.000002,1.000001,1.000002,1.000001)$^T$ & 2.8157$\times10^{-9}$ \\
 (-3.3782,2.9428,-1.8878,0.2853)$^T$  & (1.000000,1.000000,1.000000,1.000000)$^T$& 1.0359$\times10^{-10}$  \\
 (-3.9335,4.6190,-4.9537,2.7491)$^T$ &(1.000000,1.000000,1.000000,1.000000)$^T$ & 1.0346$\times10^{-10}$\\
  (3.5303,1.2206,-1.4905,0.1325)$^T$  &(1.000000,1.000000,1.000000,1.000000)$^T$ &1.0320$\times10^{-10}$ \\
 \hline
\end{tabular}$$

{\bf Remark} From the numerical results of the above example, we know that the SAVE (\ref{e11}) with finite discrete distribution can be solved by constrained optimization methods. But the EV transformation is a more complicated form with nonsmooth complementarity function and only solve SAVE (\ref{e11}) with finite discrete distribution. So, in the following section, we consider the expected residual minimization formulation, which can avoid transforming the SAVE into a complicated constrained optimization problem. And the expected residual minimization formulation can also be used to solve SAVE (\ref{e11}) with any distribution involving the finite discrete distribution.

\setcounter{proposition}{0} \setcounter{theorem}{0}
\setcounter{lemma}{0} \setcounter{corollary}{0}
\setcounter{equation}{0}
\section{Expected residual minimization formulation}

To apply the expected residual minimization formulation to solve SAVE (\ref{e11}), we first formulate the problem as the following optimization problem
\begin{equation}\label{e31}\min_{x\in R^n} F(x),\end{equation}
where $F(x)=E[\|A(\omega)x-|x|-b(\omega)\|^2]=\int_{\Omega}\|A(\omega)x-|x|-b(\omega)\|^2\rho(\omega)\diff \omega$. Discrete the involved problem by the quasi-Monte Carlo method, then the solution of the original problem can be approximated obtained by solving the discrete minimization problem.

To proceed, we give the following assumption.

\begin{assum}\label{assum31} Let $\rho:\Omega\rightarrow R_{+}$ be a continuous probability density function on probability space $(\Omega, {{\cal F}}, P)$.  Suppose that
$$\int_{\Omega}(\|A(\omega)\|+1)^2\rho(\omega)\diff \omega<\infty, ~~\int_{\Omega}\|b(\omega)\|^2\rho(\omega)\diff \omega<\infty,$$
where $A(\omega)\in R^{n\times n}$, $b(\omega)\in R^n$, $\omega \in\Omega$.\end{assum}

For $\gamma >0$, we denote the level set of function $F$ by $\Xi(\gamma)$, i.e., $$\Xi(\gamma)=\{x~|~F(x)\leq\gamma\}.$$

\begin{lemma}\label{lem31} Suppose that there exists an $\bar{\omega}\in\Omega$, such that  $\rho(\bar{\omega})>0$ and $A(\omega)\neq diag(sign(x))$. Then the level set $\Xi(\gamma)$ is bounded.\end{lemma}

\proof By $\rho$ is continuous, there exists  $\rho_0>0$ such that $$\rho(\omega)\geq \rho_0,~~\texttt{for all}~ \omega\in B=D(\bar{\omega},\delta)\cap\Omega,$$ where
$D(\bar{\omega},\delta)$ is a closed sphere with center $\bar{\omega}$ and radius $\delta$. We now consider a sequence $\{x_k\}\in R^n$. By the continuity of $\phi$, then there exists $\omega_k\in B$, such that
$$\|\phi(x_k,\omega_k)\|=\min_{\omega\in B}\|\phi(x_k,\omega)\|,$$
where  $\phi(x_k,\omega)=A(\omega)x_k-|x_k|-b(\omega)$.

Denote $\lambda=\int_{B}d\omega>0$. Then
$$\begin{array}{rl}F(x_k)\geq&\int_{B}\|\phi(x_k,\omega)\|^2\rho(\omega)\diff \omega\\
\geq&\|\phi(x_k,\omega_k)\|^2\bar{\rho}\int_{B}\diff \omega\\
=&\lambda\bar{\rho}\|\phi(x_k,\omega_k)\|^2.\end{array}$$
Now, we only need to prove $\|\phi(x_k,\omega_k)\|\rightarrow+\infty$ as $\|x_k\|\rightarrow+\infty$. Suppose $\|x_k\|\rightarrow+\infty$ holds, we know that $x_k^i\rightarrow +\infty$ or $x_k^i\rightarrow -\infty$ for some $i$. So, we get
$$((A(\omega_k)-diag(sign(x_k)))x_k-b(\omega_k))_i \rightarrow +\infty~~ \texttt{or}~~ ((A(\omega_k)-diag(sign(x_k)))x_k-b(\omega_k))_i \rightarrow -\infty$$
for some $i$, i.e., we get $\|\phi(x_k,\omega_k)\|\rightarrow+\infty$ holds for $\|x_k\|\rightarrow+\infty$. Hence, the proof is completed.\qed

In the following of this section, the quasi-Monte Carlo method for numerical integration is used as in \cite{chen05,nie92}. The transformation function $\omega=u(\nu)$ is used to go from an integral on $\Omega$ to the integral on the unit hypercube $[0,1]^m$. And the observations $\{\nu_i\}, i=1,\cdots,\bar{N}$ are generated in this unit hypercube.\\

Then, we get
$$\begin{array}{rl}F(x)=&\int_{\Omega}\|\phi(x,\omega)\|^2\rho(\omega)\diff \omega\\
=&\int_{\bar{\Omega}}\|\phi(x,u(\nu))\|^2\rho(u(\nu))u'(\nu)\diff \nu\\
=&\int_{\bar{\Omega}}\|\phi(x,u(\nu))\|^2\bar{\rho}(\nu)\diff \nu,\end{array}$$
where $\bar{\rho}(\nu)=\rho(u(\nu))u'(\nu),~\bar{\Omega}=[0,1]^m$.

For each $k$, we denote
$$F_k(x)=\frac{1}{\bar{N_k}}\sum_{\nu_i\in\bar{\Omega}_k}\|\phi(x,\nu_i)\|^2\rho(\nu_i),$$
where $\phi(x,\nu_i)=A(\nu_i)x-|x|-b(\nu_i)$, $\bar{\Omega}_k:=\{\nu_i,i=1,\cdots,\bar{N_k}\}$ is a set of observations generated by a quasi-Monte Carlo method such that $\bar{\Omega}_k\subseteq\Omega,~\bar{N_k}\rightarrow\infty$ as $k\rightarrow\infty$. In the remainder of this section, to simplify the natation, we suppose $\Omega=[0,1]^m$ and let $\omega$ denote $\nu$.

Now, we consider
\begin{equation}\label{e32}\min_{x\in R^n}F_k(x).\end{equation}
Obviously, (\ref{e32}) is the approximation problem to  (\ref{e31}).

\begin{lemma}\label{lem32}~For any fixed $x\in R^n$, we get
$$\lim_{k\rightarrow\infty}F_k(x)=F(x).$$\end{lemma}

\proof From the definition of $\phi$,  we get
$$\begin{array}{rl}\|\phi(x,\omega)\|=&\|A(\omega)x-|x|-b(\omega)\|\\
\leq&\|A(\omega)x\|+\|x\|+\|b(\omega)\|\\
\leq&(\|A(\omega)\|+1)\|x\|+\|b(\omega)\|,\end{array}$$
i.e.,
$$0\leq \|\phi(x,\omega)\|^2\leq 2[((\|A(\omega)\|+1)\|x\|)^2+\|b(\omega)\|^2].$$
By Assumption \ref{assum31}, we know that $(\|A(\omega)\|+1)^2\rho(\omega)$ is a nonnegative continuous function and it is also bounded.
Therefore, we get $\|\phi(x,\cdot)\|^2\rho(\cdot)$ is integrable and $0\leq F(x)<\infty$. By $\|\phi(x,\cdot)\|^2\rho(\cdot)$ is continuous, we have
$$\lim_{k\rightarrow\infty}F_k(x)=F(x),$$
for $x\in R^n.$ This completes the proof.\qed

Denote $\vartheta$ as the optimal solution set of (\ref{e31}), and $\vartheta_k$ as the optimal solution set of (\ref{e32}). Now, we give the following theorem to show the relation of the expected residual minimization problem (\ref{e31}) and the approximate expected residual minimization problem (\ref{e32}).

\begin{theorem}\label{Thm31}~If $\rho(\bar{\omega})>0$ holds for $\bar{\omega}\in \Omega$, then $\vartheta_k$ is nonempty and bounded when $k$ is large enough. And every accumulation point of $\{x_k\}\subseteq\vartheta_k$ is contained in the set $\vartheta$.\end{theorem}

\proof We assume that $x_k\rightarrow \bar{x}, k\rightarrow\infty$. Let $F(\overline{x})<\gamma$, by the continuity of $F$, we know that $F(x_k)\leq\gamma$
for all large $k$, i.e., $x_k\in \bar{D}(\gamma)$ for all large $k$. Now, we  show that
$$|F_k(x_k)-F_k(\overline{x})|\rightarrow0,~{\rm when ~~}k\rightarrow\infty.$$
For all $x,y\in R^n$, we get
$$\begin{array}{rl}\|\phi(x,\omega)-\phi(y,\omega)\|=&\|A(\omega)x-|x|-b(\omega)-A(\omega)y+|y|+b(\omega)\|\\
=&\|A(\omega)x-A(\omega)y+|y|-|x|\|\\
\leq&\|A(\omega)\|\|x-y\|+\|x-y\|\\
=&(\|A(\omega)\|+1)\|x-y\|.\end{array}$$
Denote $L(\omega)=\|A(\omega)\|+1$, we also get
$$\begin{array}{rl}\|\phi(x,\omega)\|\leq&(\|A(\omega)\|+1)\|x\|+\|b(\omega)\|\\
=&(\|A(\omega)\|+1)\|x\|+\|b(\omega)\|.\end{array}$$
By Lemma \ref{lem31}, $\bar{D}(\gamma)$ is closed and bounded, we have
$$\begin{array}{rl}|\|\phi(x,\omega)\|^2-\|\phi(y,\omega)\|^2|=&|(\|\phi(x,\omega)\|+\|\phi(y,\omega)\|)(\|\phi(x,\omega)\|-\|\phi(y,\omega)\|)|\\
\leq&((\|A(\omega)\|+1)\|x\|+(\|A(\omega)\|+1)\|y\|+2\|b(\omega)\|)\\
&(\|A(\omega)\|+1)\|x-y\|\\
=&(L(\omega)\|x\|+L(\omega)\|y\|+2\|b(\omega)\|)L(\omega)\|x-y\|\\
\leq&(2L(\omega)C_1+2\|b(\omega)\|)L(\omega)\|x-y\|,\end{array}$$
where $C_1=\max\{\|x\||x\in \bar{D}(\gamma)\},$ $x,y\in \bar{D}(\gamma)$.
By
$$\begin{array}{rl}(2L(\omega)C_1+2\|b(\omega)\|)L(\omega)\leq&(L(\omega)C_1+\|b(\omega)\|)^2+[L(\omega)]^2\\
=&[L(\omega)C_1]^2+\|b(\omega)\|^2+[L(\omega)]^2+2C_1L(\omega)\|b(\omega)\|\\
\leq&[L(\omega)C_1]^2+\|b(\omega)\|^2+[L(\omega)]^2+[C_1L(\omega)]^2+\|b(\omega)\|^2\\
=&(2C_1^2+1)[L(\omega)]^2+2\|b(\omega)\|^2,\end{array}$$
we obtain
$$\begin{array}{rl}|F_k(x_k)-F_k(\overline{x})|\leq&\frac{1}{\bar{N_k}}\sum\limits_{i=1}^{\bar{N_k}}|\|\phi(x_k,\omega_i)\|^2-\|\phi(\overline{x},\omega_i)\|^2|\rho(\omega_i)\\
\leq&\frac{1}{\bar{N_k}}\sum\limits_{i=1}^{\bar{N_k}}((2C_1^2+1)[L(\omega_i)]^2+2\|b(\omega_i)\|^2)\rho(\omega_i)\|x_k-\overline{x}\|\\
\leq& \digamma\|x_k-\overline{x}\|,\end{array}$$
where $\digamma$ is a constant and for all large $k$ satisfying
$$\digamma\geq \frac{1}{\bar{N_k}}\sum\limits_{i=1}^{\bar{N_k}}(((2C_1^2+1)[L(\omega_i)]^2+2\|b(\omega_i)\|^2)\rho(\omega_i)).$$
So, for $k\rightarrow \infty$,  we get
$$|F_k(x_k)-F_k(\overline{x})|\rightarrow0.$$
From the above results and Lemma \ref{lem32}, we obtain
$$|F_k(x_k)-F(\overline{x})|\leq|F_k(x_k)-F_k(\overline{x})|+|F_k(\overline{x})-F(\overline{x})|\rightarrow0,$$
when $k\rightarrow\infty$. By $x_k\in \vartheta_k$, we get
$$F_k(x_k)\leq F_k(x),~for~x\in R^n.$$
Therefore, we have
$$F(\overline{x})=\lim_{k\rightarrow\infty}F_k(x_k)\leq\lim_{k\rightarrow\infty}F_k(x)=F(x),~x\in R^n.$$
We complete the proof.\qed

Now, we give two special kinds of SAVE (\ref{e11}), which can be solved without using discrete approximation.\\

{\bf Case I.} Let $\Omega=[\tilde{\alpha}_1,\tilde{\beta}_1]\times\cdots\times[\tilde{\alpha}_N,\tilde{\beta}_N]$ with $\tilde{\alpha}_j<\tilde{\beta}_j,~j=1,\cdots,N$, and $\tilde{\omega}_j,j=1,\cdots,N$ are independent. When $\rho$ satisfies Assumption \ref{assum31} and $\rho_j$ denotes the density function for $\tilde{\omega}_j$, $j=1,\cdots,N$. We know that
$$F(x)=\sum\limits_{i=1}^nF_i(x),$$
where $$\begin{array}{l}
F_i(x)=\int_{\tilde{\alpha}_1}^{\tilde{\beta}_1}\cdots\int_{\tilde{\alpha}_N}^{\tilde{\beta}_N}[(A(\tilde{\omega})x-|x|-b(\tilde{\omega}))_i]^2\rho_1(\tilde{\omega}_1)\cdots\rho_N(\tilde{\omega}_N)\diff \tilde{\omega}_1\cdots \diff \tilde{\omega}_N.\end{array}$$

{\bf Case II.} Let $A(\omega)\equiv A$ and $b(\omega)\equiv\tilde{b}+T\omega$, where $A\in R^{n\times n}$, $\tilde{b}\in R^n$ and $T\in R^{n\times m}$ are given constants. For each $i$, the $i$th row of the matrix $T$ has just one positive element $t_i$, and the density function $\rho$ is defined by
$$\rho(\omega)=\left\{
\begin{array}{rcl}
1,~~\omega\in[0,1]^m\\
0,~~otherwise.\\
\end{array}\right.$$
In this case, we get $F(x)=\sum\limits_{i=1}^nF_i(x)$, where $$\begin{array}{rl}F_i(x)=&\int_0^1\cdots\int_0^1[(Ax-|x|-b(\omega))_i]^2\rho_1(\omega_1)\rho_2(\omega_2)\cdots\rho_m(\omega_m)\diff \omega_1\diff \omega_2\cdots \diff \omega_m\\
=&\int_0^1\cdots\int_0^1\int_0^1[(Ax-|x|-\tilde{b}-t_i\omega_j)_i]^2\rho_j(\omega_j)\diff \omega_j\rho_1(\omega_1)\cdots\rho_{j-1}(\omega_{j-1})\\
&\rho_{j+1}(\omega_{j+1})\cdots\rho_m(\omega_m)\diff \omega_1\cdots \diff \omega_{j-1}\diff \omega_{j+1}\cdots \diff \omega_m\\
=&[Ax-|x|-\tilde{b}]_i^2+\frac{1}{3}t_i^2-[Ax-|x|-\tilde{b}]_it_i.\end{array}$$

\setcounter{proposition}{0} \setcounter{theorem}{0}
\setcounter{lemma}{0} \setcounter{corollary}{0}
\setcounter{equation}{0}
\section{A smoothing gradient method}
In this section, we use the ERM formulation to transform SAVE (\ref{e11}) into an unconstrained optimization problem. For SAVE (\ref{e11}) contains nonsmooth term $|x|$, we consider smoothing method to solve it. Smoothing gradient method is an effective smoothing method to deal with this kind of problems \cite{BLO05,K07,C12}, so we use the smoothing gradient method to solve SAVE (\ref{e11}).

Firstly, we generate samples $\omega^i, i=1,2,\cdots,N$, i.e.,
\begin{equation*}
f(x)=\frac{1}{N}\sum_{i=1}^N\|A(\omega_i)x-|x|-b(\omega_i)\|^2\rho(\omega_i),
\end{equation*}
and we choose the smoothing function of $|x_i|$ as
\begin{equation*}
\psi(x_i,\mu)=\sqrt{x_i^2+\mu}, i=1,2,\cdots,n,
\end{equation*}
where $\mu\geq 0$. Denote $\psi(x,\mu)=(\sqrt{x_1^2+\mu}, \sqrt{x_2^2+\mu}, \cdots, \sqrt{x_n^2+\mu})^{T}$.
Then SAVE (\ref{e11}) can be transformed into the following unconstrained optimization problem
\begin{equation}\label{ghmin}
\underset {x\in R^n}{\min}\;\widetilde{f}(x,\mu)=\frac{1}{N}\sum_{i=1}^N\|A(\omega_i)x-\psi(x,\mu)-b(\omega_i)\|^2\rho(\omega_i).
\end{equation}
And the gradient of the objective function in \eqref{ghmin} is
\begin{equation*}
\nabla_{x}{\widetilde{f}(x,\mu)}=\frac{2}{N}\sum_{i=1}^NJ(A(\omega_i)x-\psi(x,\mu)-b(\omega_i))^{T}(A(\omega_i)x-\psi(x,\mu)-b(\omega_i))\rho(\omega_i),
\end{equation*}
where $J(A(\omega_i)x-\psi(x,\mu)-b(\omega_i))$ is the Jacobian of $(A(\omega_i)x-\psi(x,\mu)-b(\omega_i))$.

Next, we give the smoothing gradient method for SAVE (\ref{e11}).
\begin{algorithm}[H]\label{PH}
\caption{Smoothing gradient method}
\begin{algorithmic}[1]
\item[{\bf Step 0}]Given an initial point $x_0\in R^n$, $\mu_0\in R$, $\sigma,\delta,\rho,\bar{\gamma}\in(0,1)$, $\epsilon>0$, set $k=0$.

\item[{\bf Step 1}]If $\|\nabla_{x}{\widetilde{f}(x_k,\mu_k)}\|\leq\varepsilon$, stop. Otherwise, go to Step 2.

\item[{\bf Step 2}]Computing the search direction $d_k=-\nabla{\widetilde{f}(x_k,\mu_k)}$.
\item[{\bf Step 3}] Determine $\alpha_k=\underset j{\max}\{\rho^j, j=0,1,2,\cdots\}$ satisfying
\begin{equation}\label{armijo}
\widetilde{f}(x_{k+1},\mu_k)-\widetilde{f}(x_k,\mu_k)\leq\delta\alpha_k\nabla_{x}{\widetilde{f}(x_k,\mu_k)}^{T}d_k.
\end{equation}
Set $x_{k+1}=x_k+\alpha_k d_k$.

\item[{\bf Step 4}]If $\|\nabla{\widetilde{f}(x_k,\mu_k)}\|\geq\bar{\gamma}\mu_k$, then set $\mu_{k+1}=\mu_k$; Otherwise choose $\mu_{k+1}=\sigma \mu_k$.

\item[{\bf Step 5}]Let $k=k+1$ and return to Step 1.

\end{algorithmic}
\end{algorithm}

It is easy to find that $\widetilde{f}(\cdot,\mu)\geq 0$ for any constant $\mu\geq0$ and $\forall x\in R^n$, $\nabla_{x}{\widetilde{f}(\cdot,\mu)}$ is uniformly continuous on the level set $L(x_0,\mu)=\{x\in R^n|\widetilde{f}(x,\mu)\leq\widetilde{f}(x_0,\mu)\}$. Next, we give the global convergence of the proposed smoothing gradient method.
\begin{lemma}\label{lemma4.1}
Given a constant $\overline{\mu}\geq 0$. Let $\{x_k\}$ be the sequence generated by Algorithm 1, then
\begin{equation*}
\underset{k\rightarrow\infty}{\lim}\|\nabla_{x}{\widetilde{f}(x_{k},\overline{\mu})}\|=0.
\end{equation*}
\end{lemma}
\proof We proof this lemma by contradiction. Set $\underset{k\rightarrow\infty}{\lim}x_k=x^*$. Suppose that $\nabla_{x}{\widetilde{f}(x^*,\overline{\mu})}\neq 0$. From the continuity of $\widetilde{f}(x_k,\overline{\mu})$ and $\nabla_{x}{\widetilde{f}(x_{k},\overline{\mu})}$, we have
\begin{equation*}
\underset{k\rightarrow\infty}{\lim}\widetilde{f}(x_k,\overline{\mu})=\widetilde{f}(x^*,\overline{\mu}),
\end{equation*}
\begin{equation*}
\underset{k\rightarrow\infty}{\lim}\widetilde{f}(x_k,\overline{\mu})-\underset{k\rightarrow\infty}{\lim}\widetilde{f}(x_{k+1},\overline{\mu})=0.
\end{equation*}

By \eqref{armijo}, let $\alpha_k=\rho^{j_k}$, where $j_k$ is the smallest non-negative integer satisfying the inequality \eqref{armijo}. Combining with $s_k=\rho^{j_k}d_k$, we have
\begin{equation*}
\underset{k\rightarrow\infty}{\lim}(-\delta\nabla_x\widetilde{f}(x_k,\overline{\mu})^{T}s_k)=0,
\end{equation*}
so,
\begin{equation*}
\underset{k\rightarrow\infty}{\lim}\nabla_x\widetilde{f}(x_k,\overline{\mu})^{T}s_k=0.
\end{equation*}
From $\underset{k\rightarrow\infty}{\lim}\nabla_x\widetilde{f}(x_k,\overline{\mu})\neq0$, and then $\underset{k\rightarrow\infty}{\lim}\|s_k\|=0$. Due to $j_k$ is the smallest non-negative integer satisfying the inequality \eqref{armijo}, set $\rho^{j_k-1}=\frac{\rho^{j_k}}{\rho}$, then
\begin{equation*}
\widetilde{f}(x_k+\rho^{j_k-1}d_k,\overline{\mu})-\widetilde{f}(x_k,\overline{\mu})>\delta\rho^{j_k-1}\nabla_x\widetilde{f}(x_k,\overline{\mu})^{T}d_k.
\end{equation*}
By $\rho^{j_k-1}d_k=\frac{\rho^{j_k}d_k}{\rho}=\frac{s_k}{\rho}$, we get
\begin{equation}\label{skp}
\widetilde{f}(x_k+\frac{s_k}{\rho},\overline{\mu})-\widetilde{f}(x_k,\overline{\mu})>\delta\nabla_x\widetilde{f}(x_k,\overline{\mu})^{T}\frac{s_k}{\rho}.
\end{equation}
Let $p_k=\frac{s_k}{\|s_k\|}$, then $\frac{s_k}{\rho}=\frac{\|s_k\|}{\rho}p_k$, from $\underset{k\rightarrow\infty}{\lim}\|s_k\|=0$, we know that
\begin{equation*}
\underset{k\rightarrow\infty}{\lim}{\alpha_k}'=\underset{k\rightarrow\infty}{\lim}\frac{\|s_k\|}{\rho}=0.
\end{equation*}
\eqref{skp} can be rewritten as
\begin{equation}\label{frac}
\frac{\widetilde{f}(x_k+{\alpha_k}'p_k,\overline{\mu})-\widetilde{f}(x_k,\overline{\mu})}{{\alpha_k}'}>\delta\nabla_x\widetilde{f}(x_k,\overline{\mu})^{T}p_k.
\end{equation}
Due to $\|p_k\|=1$, the sequence $\{\|p_k\|\}$ is bounded. So there exists a convergence subsequence, let $\underset{k\rightarrow\infty}{\lim}\|p_k\|=p^*$, and then we get $p^*=1$. Taking the limit of both sides of \eqref{frac}, we have
\begin{equation*}
-\nabla_x\widetilde{f}(x^*,\overline{\mu})^{T}p^*\geq\delta\nabla_x\widetilde{f}(x^*,\overline{\mu})^{T}p^*.
\end{equation*}
Since $\delta\in(0,1)$, we obtain
\begin{equation}\label{tid0}
\nabla_x\widetilde{f}(x^*,\overline{\mu})^{T}p^*\geq 0.
\end{equation}
In addition, $p_k=\frac{s_k}{\|s_k\|}=\frac{d_k}{\|d_k\|}$, thus
\begin{eqnarray}\label{bdengsh}
  -\nabla_x\widetilde{f}(x_k,\overline{\mu})^{T}p_k &=& -\nabla_x\widetilde{f}(x_k,\overline{\mu})^{T}(\frac{d_k}{\|d_k\|}) \nonumber\\
   &=& \|\nabla_x\widetilde{f}(x_k,\overline{\mu})\|cos\theta_k \nonumber\\
   &\geq& \|\nabla_x\widetilde{f}(x_k,\overline{\mu})\|sin\beta,
\end{eqnarray}
where $\theta_k$ is the angle between $d_k$ and $-\nabla_x\widetilde{f}(x_k,\overline{\mu})$, $\beta\in(0,\frac{\pi}{2})$. Taking the limit of both sides of the inequation \eqref{bdengsh}, we have
\begin{equation*}
-\nabla_x\widetilde{f}(x^*,\overline{\mu})^{T}p^*\geq\|\nabla_x\widetilde{f}(x^*,\overline{\mu})\|sin\beta>0,
\end{equation*}
which contradicts with \eqref{tid0}. So we have $\underset{k\rightarrow\infty}{\lim}\nabla_x\widetilde{f}(x_k,\overline{\mu})=0$. The proof is completed.

\begin{theorem}\label{Thm41}
Let $\{\mu_k\}$ and $\{x_k\}$ be the sequence generated by Algorithm 1. Then
\begin{equation*}\label{ft>}
\underset{k\rightarrow\infty}{\lim}\|\nabla_{x}{\widetilde{f}(x_{k+1},\mu_k)}\|=0.
\end{equation*}
\end{theorem}
\proof
Define $K=\{k|\mu_{k+1}=\sigma \mu_k\}$. Suppose $K$ is a finite set, then there exists an integer $\widehat{k}$ such that for all $k>\widehat{k}$,
\begin{equation}\label{fgm}
\|\nabla_{x}{\widetilde{f}(x_{k},\mu_{k-1})}\|\geq \bar{\gamma} \mu_k,
\end{equation}
set $\mu_k=\mu_{\widehat{k}}=\overline{\mu}$, we get
\begin{center}
$\underset{x\in\mathbb{R}^n}{\min}\;\widetilde{f}(x,\overline{\mu})$.
\end{center}
From Lemma \ref{lemma4.1}, we have
\begin{equation*}
\underset{k\rightarrow\infty}{\lim}\|\nabla_{x}\widetilde{f}(x_k,\overline{\mu})\|=0,
\end{equation*}
which contradicts with \eqref{fgm}. So $K$ is an infinite set, i.e., $\underset{k\rightarrow\infty}{\lim}\mu_k=0$. Set $K=\{k_0, k_1, \cdots\}$, for $k_0<k_1<\cdots$, then
\begin{equation*}
\underset{k\rightarrow\infty}{\lim}\|\nabla_{x}\widetilde{f}(x_{k_i+1},\mu_{k_i})\|\leq\bar{\gamma}\underset{i\rightarrow\infty}{\lim}\mu_{k_i}=0.
\end{equation*}
The proof is completed.

In the following of this section, we verify the effectiveness of Algorithm 1 via the following given examples. The parameters used in Algorithm 1 are chosen as $\rho=0.5$, $\sigma=0.5$, $\delta=0.5$, $\mu_0=0.01$, $\bar{\gamma}=0.5$. We terminate our algorithm if $k\geq 10000$ or $\|\nabla_{x}\widetilde{f}(x_{k},\mu_{k})\|<1.0e-5$ satisfied. We implement all numerical test in MATLAB R2019b and run the codes on a PC with 1.80GHz CPU.
\begin{example}~Consider SAVE (\ref{e11}), where
$$A(\omega)=\left(
\begin{array}{rcl}
2+\omega  &  1\\
5  &  1+\omega
\end{array}\right),~~b(\omega)=\left(
\begin{array}{rcl}
4+\omega\\
5+3\omega
\end{array}\right).$$
\end{example}

We generate samples $\omega^i$, $i=1,2,\cdots,N$, which obey the uniform distribution of $[0,1]$. The numerical results of Example 4.1 are shown in Table 4.1 and Figure 4.1, where $N$ denotes the number of $\omega^i$, $x^0$, $x^*$ and $f(x^*)$ denote the initial point, the optimum solution and the optimum value, respectively.\\
\newpage
$$\mbox{Table 4.1 Numerical results of Example 4.1}$$
$$\begin{tabular}{cccc}
\hline \hline \noalign{\smallskip}
$N$  &  $x^0$  &  $x^*$  &  $f(x^*)$\\
\hline
10 & (0.9415,1.7138)$^T$ & (1.0000,3.0000)$^T$ &1.2332e-09\\
50 & (1.5088,0.6925)$^T$  & (1.0000,3.0000)$^T$ & 1.2342e-09 \\
 100 & (1.6206,1.1140)$^T$  & (1.0000,3.0000)$^T$& 1.2553e-09 \\
 200 & (1.6822,0.7090)$^T$ &(1.0000,3.0000)$^T$ & 1.2360e-09\\
 500 &(1.3098,1.7802)$^T$  &(1.0000,3.0000)$^T$ & 1.2104e-09\\
\noalign{\smallskip}\hline
\end{tabular}$$

\begin{figure}[h]
    \centering
    \includegraphics[width=8cm,height=6cm]{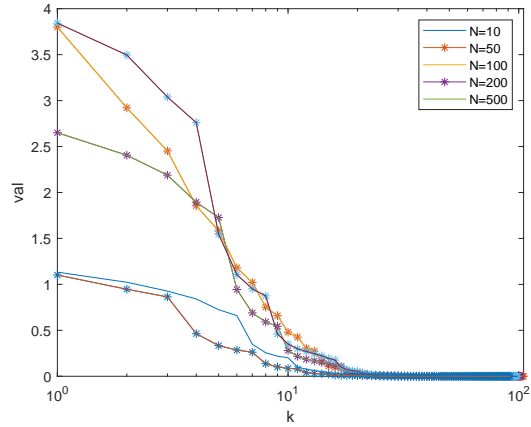}
    \captionsetup{font={scriptsize}}
    \caption{Numerical results of Example 4.1}
\end{figure}

\begin{example}~Consider SAVE (\ref{e11}), where
\begin{center}
$A(\omega)=\begin{pmatrix}2+\omega&1&0&0\\2&1+\omega&0&0\\0&0&2+\omega&1\\0&2&0&1+\omega\end{pmatrix}$, $b(\omega)=\begin{pmatrix}2+\omega\\2+\omega\\2+\omega\\2+\omega\end{pmatrix}$.
\end{center}
\end{example}

We generate samples $\omega^i$, $i=1,2,\cdots,N$, which obey the uniform distribution of $[0,1]$. The detailed numerical results are shown in Table 4.2 and Figure 4.2, where $N$ denotes the number of $\omega^i$, $x^0$, $x^*$ and $f(x^*)$ denote the initial point, the optimum solution and the optimum value, respectively.\\
\newpage
$$\mbox{Table 4.2 Numerical results of Example 4.2}$$
$$\begin{tabular}{cccc}
\hline \hline \noalign{\smallskip}
$N$  &  $x^0$  &  $x^*$  &  $f(x^*)$\\
\hline
10 & (1.3027,1.4874,0.6039,0.1792)$^T$ & (1.0000,1.0000,1.0000,1.0000)$^T$ &5.7084e-09\\
50 & (1.0894,1.9952,1.0220,1.7470)$^T$  & (1.0000,1.0000,1.0000,1.0000)$^T$ & 5.7252e-09 \\
 100 & (0.9878,1.7254,0.4858,1.6685)$^T$  & (1.0000,1.0000,1.0000,1.0000)$^T$& 5.8684e-09 \\
 200 & (0.2891,0.7410,1.2448,1.9951)$^T$ &(1.0000,1.0000,1.0000,1.0000)$^T$ & 5.7938e-09\\
 500 &(1.6171,1.9691,1.7718,0.4277)$^T$  &(1.0000,1.0000,1.0000,1.0000)$^T$ & 5.7870e-09\\
\noalign{\smallskip}\hline
\end{tabular}$$

\begin{figure}[h]
    \centering
    \includegraphics[width=8cm,height=6cm]{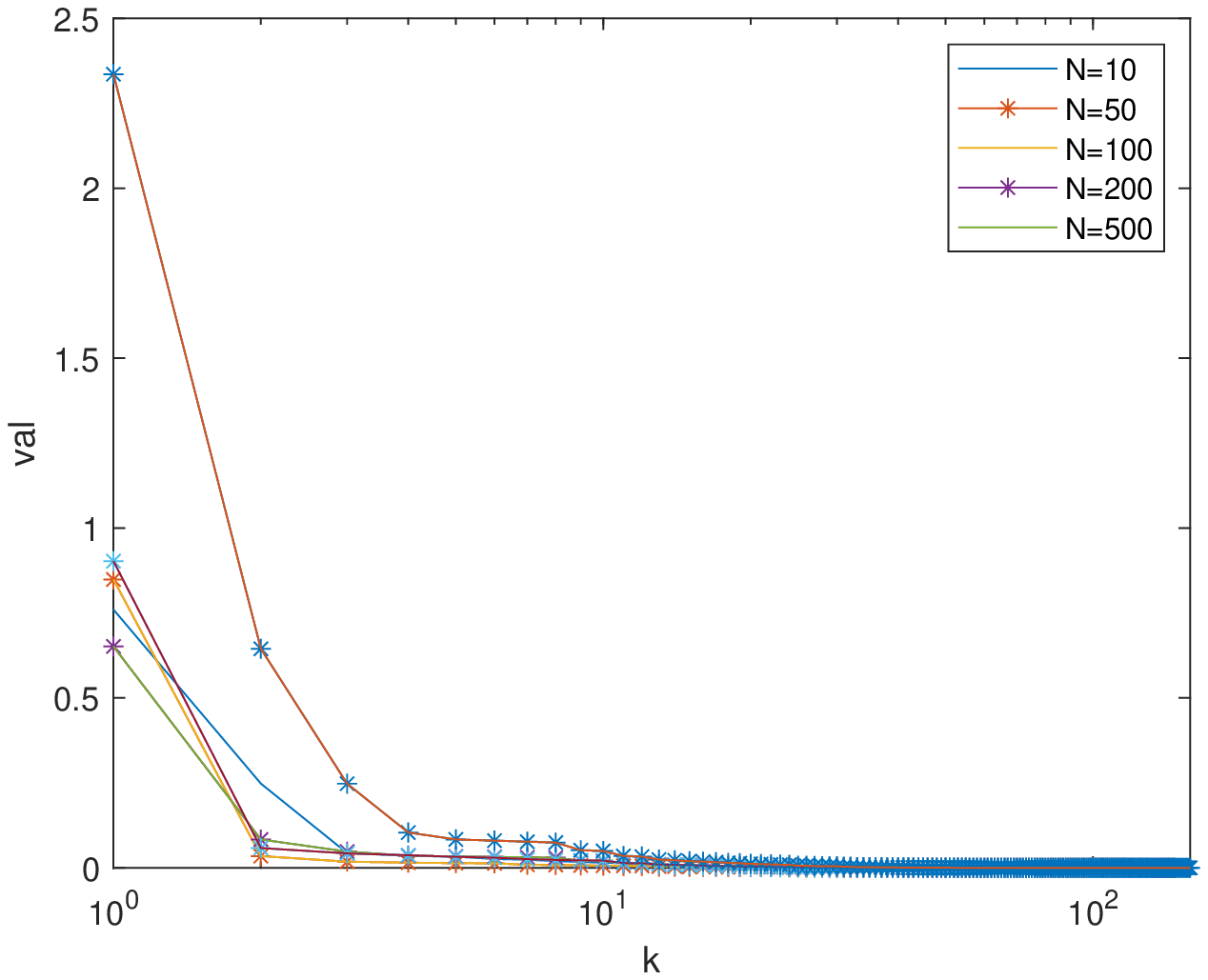}
    \captionsetup{font={scriptsize}}
    \caption{Numerical results of Example 4.2}
\end{figure}

\begin{example}~Consider SAVE (\ref{e11}), where
\begin{small}
\begin{equation*}
A(\omega)=\begin{pmatrix}5+\omega&0&0&0&0&2&1&0&0&3\\ \frac{1}{2}&2+\omega&0&\frac{1}{2}&1&0&1&0&6&0\\0&\frac{1}{4}&7+\omega&\frac{3}{4}&0&2&0&0&\frac{1}{2}&\frac{1}{2}\\1&1&2&2+\omega&\frac{1}{2}&0&\frac{3}{2}&2&0&1\\0&0& \frac{2}{5}&\frac{1}{4}&6+\omega&2&0&1&\frac{7}{20}&1\\2&\frac{1}{2}&4&0&0&1+\omega&\frac{1}{2}&2&1&0\\0&5&0&\frac{2}{3}&0&\frac{2}{3}&3+\omega&\frac{1}{4}&1&\frac{5}{12}\\
2&1&1&1&1&\frac{1}{2}&0&4+\omega&\frac{1}{2}&0\\ \frac{1}{7}&\frac{5}{7}&0&0&1&0&\frac{1}{7}&0&9+\omega&0\\3&0&2&1&\frac{5}{2}&0&\frac{1}{2}&\frac{1}{4}&\frac{1}{4}&1+\omega\end{pmatrix},
\end{equation*}
\end{small}
and
\begin{center}
$b(\omega)=(10+\omega,10+\omega,\cdots,10+\omega)^{T}\in R^{10}$.
\end{center}
\end{example}

We generate samples $\omega^i$, $i=1,2,\cdots,N$, which obey the uniform distribution of $[0,1]$. The initial points are randomly generated. The detailed numerical results are shown in Table 4.3 and Figure 4.3, where $N$ denotes the number of $\omega^i$, $x^*$ and $f(x^*)$ denote the optimum solution and the optimum value respectively.
$$\mbox{Table 4.3 Numerical results of Example 4.3}$$
$$\begin{tabular}{cccc}
\hline \hline \noalign{\smallskip}
$N$   &  $x^*$  &  $f(x^*)$\\
\hline
10  & (1.0858,1.0705,$\cdots$,1.0122)$^T$ &0.0063\\
50  & (1.0882,1.0776,$\cdots$,1.0097)$^T$ & 0.0072 \\
 100  & (1.0893,1.0727,$\cdots$,0.9988)$^T$& 0.0086 \\
 200  &(1.0890,1.0725,$\cdots$,0.9998)$^T$ & 0.0084\\
 500   &(1.0892,1.0737,$\cdots$,1.0001)$^T$ & 0.0084\\
\noalign{\smallskip}\hline
\end{tabular}$$

\begin{figure}[h]
    \centering
    \includegraphics[width=8cm,height=6cm]{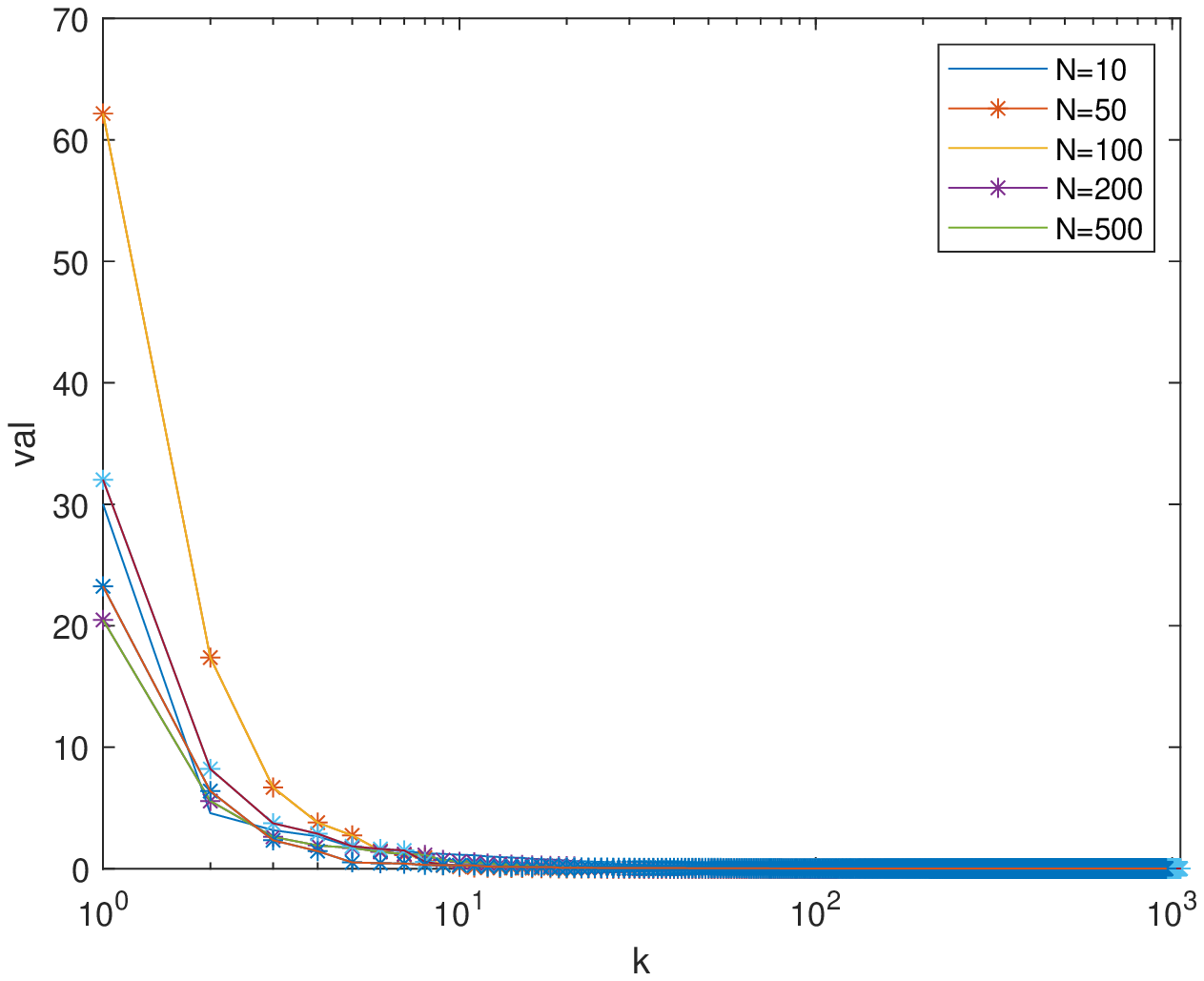}
    \captionsetup{font={scriptsize}}
    \caption{Numerical results of Example 4.3}
\end{figure}

\begin{example}~Consider SAVE (\ref{e11}), where
\begin{center}
$A(\omega)={\begin{pmatrix}2+\omega&1&&&&\\1&2+\omega&1&&&\\&1&2+\omega&1&&\\&&\ddots&\ddots&\ddots&\\&&&1&2+\omega&1\\&&&&1&2+\omega\end{pmatrix}}_{n\times n}$,
\end{center}
and
\begin{center}
$b(\omega)=(2+\omega,3+\omega,3+\omega,\cdots,3+\omega,2+\omega)^{T}\in R^n$.
\end{center}
\end{example}

We generate samples $\omega^i$, $i=1,2,\cdots,N$, which obey the uniform distribution of $[0,1]$. The initial points are randomly generated. The detailed numerical results are shown in Table 4.4, Table 4.5, Figure 4.4 and Figure 4.5, where $n$ denotes the dimension, $N$ denotes the number of $\omega^i$, $x^*$ and $f(x^*)$ denote the optimum solution and the optimum value respectively.
$$\mbox{Table 4.4 Numerical results of Example 4.4 (n=100)}$$
$$\begin{tabular}{cccc}
\hline \hline \noalign{\smallskip}
 $N$   &  $x^*$  &  $f(x^*)$\\
\hline
10  & (1.0000,1.0000,$\cdots$,1.0000)$_{100\times 1}^T$ & 1.4082e-07 \\
50  & (1.0000,1.0000,$\cdots$,1.0000)$_{100\times 1}^T$ & 1.4026e-07 \\
 100  & (1.0000,1.0000,$\cdots$,1.0000)$_{100\times 1}^T$ & 1.3966e-07 \\
 200  & (1.0000,1.0000,$\cdots$,1.0000)$_{100\times 1}^T$ & 1.4032e-07\\
 500   & (1.0000,1.0000,$\cdots$,1.0000)$_{100\times 1}^T$ & 1.4027e-07\\
\noalign{\smallskip}\hline
\end{tabular}$$

\begin{figure}[h]
    \centering
    \includegraphics[width=8cm,height=6cm]{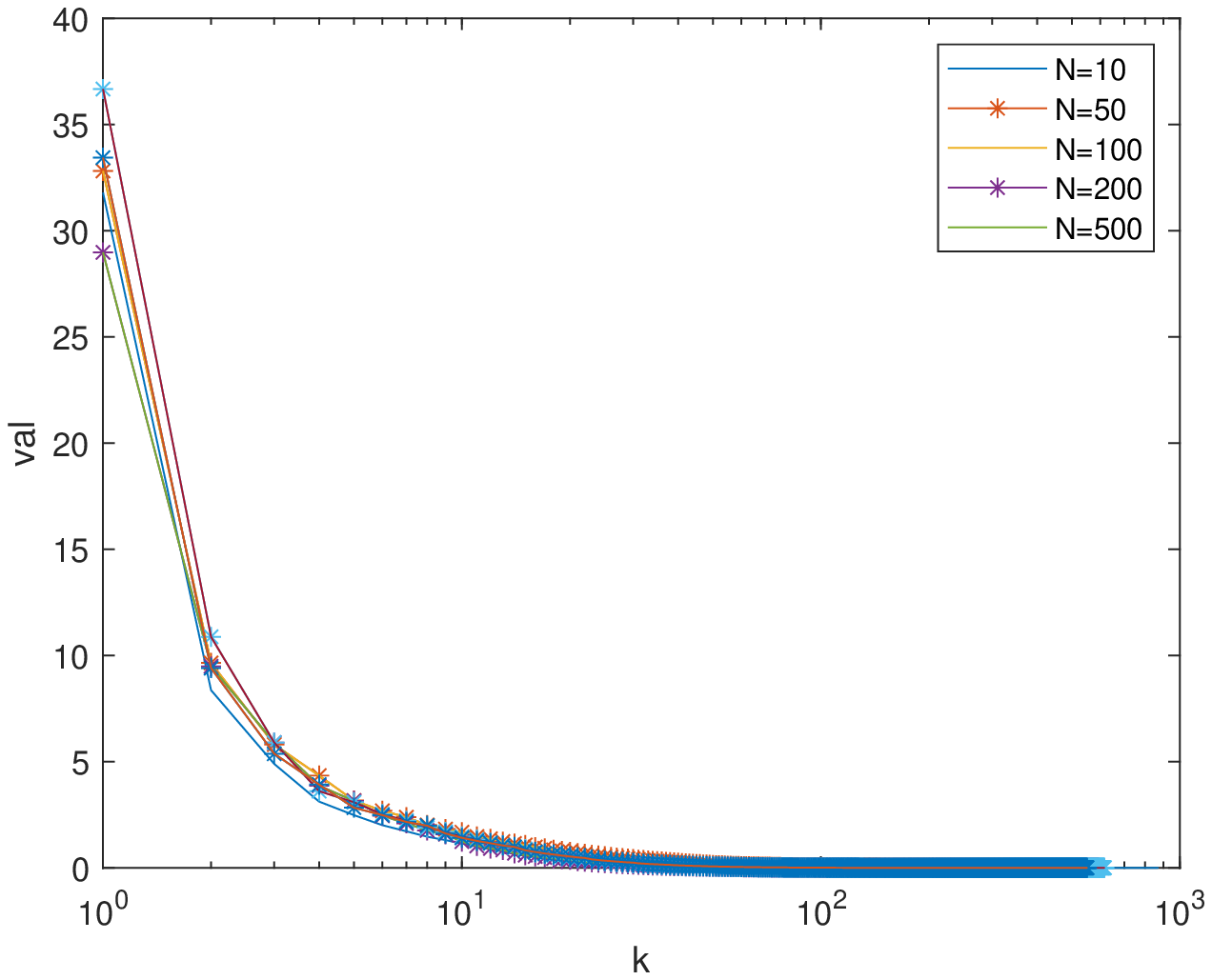}
    \captionsetup{font={scriptsize}}
    \caption{Numerical results of Example 4.4 (n=100)}
\end{figure}

\newpage
$$\mbox{Table 4.5 Numerical results of Example 4.4 (n=500)}$$
$$\begin{tabular}{cccc}
\hline \hline \noalign{\smallskip}
 $N$   &  $x^*$  &  $f(z^*)$\\
\hline
10  & (1.0000,1.0000,$\cdots$,1.0000)$_{500\times 1}^T$ & 7.0087e-07 \\
50  & (1.0000,1.0000,$\cdots$,1.0000)$_{500\times 1}^T$ & 6.9867e-07 \\
 100  & (1.0000,1.0000,$\cdots$,1.0000)$_{500\times 1}^T$ & 6.9962e-07 \\
 200  & (1.0000,1.0000,$\cdots$,1.0000)$_{500\times 1}^T$ & 6.9930e-07\\
 500   & (1.0000,1.0000,$\cdots$,1.0000)$_{500\times 1}^T$ & 6.9897e-07\\
\noalign{\smallskip}\hline
\end{tabular}$$

\begin{figure}[h]
    \centering
    \includegraphics[width=8cm,height=6cm]{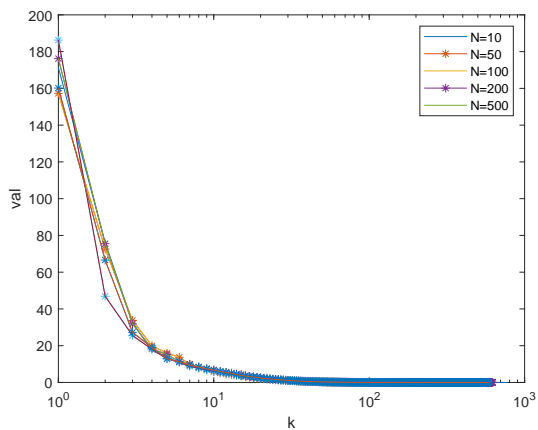}
    \captionsetup{font={scriptsize}}
    \caption{Numerical results of Example 4.4 (n=500)}
\end{figure}

From the above numerical results, we can see that SAVE (\ref{e11}) can be solved by simple unconstrained optimization method. And the ERM formulation can avoid transforming SAVE (\ref{e11}) into a complicated constrained optimization problem.

\section{Conclusions} In this paper, we propose a new kind of absolute value equation problem with random quantities, which is called stochastic absolute value equations. The properties of the proposed stochastic absolute value equations are studied. The expected value formulation and expected residual minimization formulation for solving the proposed stochastic absolute value equations are also given. Absolute value equations is widely used in studying engineering problems, economics and management problems. It is very meaningful to study this kind of stochastic absolute value equations, which is a more extensive problem than the absolute value equations.\\

{\bf Acknowledgements.} This work was supported by National Natural Science Foundation of China (No.11671220, 12171271). The authors wish to give their sincere thanks to the editor and the anonymous referees for their valuable and helpful comments, which help improve the quality of this paper significantly.\\

\end{document}